\documentclass{amsart}
 \usepackage{amssymb, latexsym}
 \usepackage{url}

\DeclareMathOperator{\var}{\mathrm{Var}}
\DeclareMathOperator{\cov}{\mathrm{Cov}}

\begin{document}
 \bibliographystyle{plain}

 \newtheorem*{KLT}{The Khinchin-Levy Theorem}
 \newtheorem*{KT}{Khinchin's Theorem}
 \newtheorem{theorem}{Theorem}
 \newtheorem{lemma}{Lemma}
 \newtheorem{corollary}{Corollary}
 \newtheorem{conjecture}{Conjecture}
 \newtheorem{definition}{Definition}
 \newcommand{\mc}{\mathcal}
 \newcommand{\F}{\mc{F}}
 \newcommand{\FN}{\F_n}
 \newcommand{\fcap}{\F\cap (0,1)}
 \newcommand{\rar}{\rightarrow}
 \newcommand{\Rar}{\Rightarrow}
 \newcommand{\lar}{\leftarrow}
 \newcommand{\lrar}{\leftrightarrow}
 \newcommand{\Lrar}{\Leftrightarrow}
 \newcommand{\zpz}{\mathbb{Z}/p\mathbb{Z}}
 \newcommand{\mbb}{\mathbb}
 \newcommand{\A}{\mc{A}}
 \newcommand{\B}{\mc{B}}
 \newcommand{\I}{\mc{I}}
 \newcommand{\J}{\mc{J}}
 \newcommand{\D}{\mc{D}}
 \newcommand{\E}{\mc{E}}
 \newcommand{\U}{\mc{U}}
 \newcommand{\V}{\mc{V}}
 \newcommand{\W}{\mc{W}}
 \newcommand{\itQ}{\mc{Q}}
 \newcommand{\C}{\mathbb{C}}
 \newcommand{\R}{\mathbb{R}}
 \newcommand{\N}{\mathbb{N}}
 \newcommand{\Q}{\mathbb{Q}}
 \newcommand{\Z}{\mathbb{Z}}
 \newcommand{\fsum}{\sum_{\beta\in\F\cap [0,1)}}
 \newcommand{\fcapsum}{\sum_{\beta\in\fcap}}
 \newcommand{\qsum}{\sum_{q=1}^Q}
 \newcommand{\psum}{\sum_{\substack{p=1\\(p,q)=1}}^q}
 \newcommand{\asum}{\sum_{\substack{a=1\\(a,q)=1}}^q}
 \newcommand{\ponesum}{\sum_{\substack{p_1=1\\(p_1,q_1)=1}}^{q_1}}
 \newcommand{\ptwosum}{\sum_{\substack{p_2=1\\(p_2,q_2)=1}}^{q_2}}
 \newcommand{\fb}{f_{\beta}}
 \newcommand{\fg}{f_{\gamma}}
 \newcommand{\gb}{g_{\beta}}
 \newcommand{\nsb}{N_{\text{SB}}}
 \newcommand{\vphi}{\varphi}
 \newcommand{\Xqn}{X_{q,N}}
 \newcommand{\wh}{\widehat}
 \newcommand{\whXq}{\widehat{X}_q(0)}
 \newcommand{\whXqN}{\widehat{X}_{q,N}(0)}
 \newcommand{\whYq}{\widehat{Y}_q(0)}
 \newcommand{\Xnn}{g_{n,N}}
 \newcommand{\whXn}{\widehat{X}_{n,N}(0)}
 \newcommand{\aqb}{a_{q,\beta}}
 \newcommand{\lf}{\left\lfloor}
 \newcommand{\rf}{\right\rfloor}
 \newcommand{\lQx}{L_Q(x)}
 \newcommand{\lQQ}{\frac{\lQx}{Q}}
 \newcommand{\rQx}{R_Q(x)}
 \newcommand{\rQQ}{\frac{\rQx}{Q}}
 \newcommand{\elQ}{\ell_Q(\alpha )}
 \newcommand{\oa}{\overline{a}}

\title[Intermediate convergents and a metric theorem of Khinchin]{Intermediate convergents and\\ a metric theorem of Khinchin}
\author{Alan K. Haynes}
\thanks{Research supported by EPSRC grant EP/F027028/1}
\address{Department of Mathematics, University of York, Heslington, York YO10 5DD, UK}
\email{akh502@york.ac.uk}
 \allowdisplaybreaks


\begin{abstract}
A landmark theorem in the metric theory of continued fractions
begins this way: Select a non-negative real function $f$ defined
on the positive integers and a real number $x$, and form the
partial sums $s_n$ of $f$ evaluated at the partial quotients
$a_1,\ldots ,a_n$ in the continued fraction expansion for $x$.
Does the sequence $\{s_n/n\}$ have a limit as $n\rar\infty$? In
1935 A.~Y.~Khinchin proved that the answer is yes for almost
every $x$, provided that the function $f$ does not grow too
quickly. In this paper we are going to explore a natural
reformulation of this problem in which the function $f$ is
defined on the rationals and the partial sums in question are
over the intermediate convergents to $x$ with denominators less
than a prescribed amount. By using some of Khinchin's ideas
together with more modern results we are able to provide a
quantitative asymptotic theorem analogous to the classical one
mentioned above.
\end{abstract}


\maketitle

\section{Definitions and Statement of Results}
For each real number $x$ we denote the simple continued fraction
expansion of $x$ by
\begin{align*} x=a_0 + \cfrac{1}{a_1+
            \cfrac{1}{a_2+\dotsb}}=[a_0;a_1,a_2,\ldots],
\end{align*}
where $a_0\in\Z$ and $a_n\in\N$ for each $n\ge 1$. The integers
$a_n, n\ge 0$ are the {\it partial quotients} in the continued
fraction expansion of $x$. If $x$ is irrational then this
expansion is unique. If $x\in\Q\setminus\Z$ then there are only
finitely many nonzero partial quotients in its continued fraction
expansion, and we ensure that the expansion is unique by
requiring that the last nonzero partial quotient be greater than
one. Finally if $x\in\R/\Z$ then we ensure that its continued
fraction expansion is unique by requiring that $a_0=0$.

In 1935 A.~Y.~Khinchin published a proof of the following result
\cite{khinchin1935}, \cite[Theorem 35]{khinchin1964}.
\begin{KT}
Suppose that $f(r)$ is a non-negative function of a natural
argument $r$ and suppose that there exist positive constants $C$
and $\delta$ such that
\[f(r)<Cr^{1/2-\delta}\qquad
(r=1,2,\ldots).\]
Then, for all numbers in the interval $(0,1)$,
with the exception of a set of measure zero,
\[\lim_{n\rar\infty}\frac{1}{n}\sum_{k=1}^nf(a_k)=\sum_{r=1}^\infty f(r)\frac{\log\left(1+\frac{1}{r(r+2)}\right)}{\log 2}.\]
\end{KT}

In this paper we wish to formulate a natural variant of this
result, which we present as Theorem \ref{MQestimate2} below. Our
variant can be seen as a refinement of Khinchin's Theorem and it
also raises several interesting and apparently difficult open
questions, one of which we discuss at the end of this section.

To this end we introduce for $x\in\R$ and $n\ge 0$ the $n$th {\it
principal convergent} to $x$
\[\frac{p_n}{q_n}=[a_0;a_1\ldots a_n],\]
which we will always assume is written in lowest terms. For
$n=-2,-1$ we also define $p_{-2}=q_{-1}=1$ and $p_{-1}=q_{-2}=0$.
For $n\ge 1$ we define the subset $E_n=E_n(x)$ of $\Q$ by
\begin{align*}
E_n &= \left\{\frac{mp_{n-1} + p_{n-2}}{mq_{n-1} + q_{n-2}}: m =
1,
2, \dots , a_n\right\}\\
&= \big\{[a_0; a_1, a_2, \dots , a_{n-1}, m]: m = 1, 2, \dots,
a_n\big\}.
\end{align*}
Each set $E_n$ contains $a_n$ distinct fractions, including the
principal convergent $p_n/q_n$.  The remaining fractions (if any)
indexed by $m = 1, 2, \dots , a_n-1$ are the {\it intermediate
convergents} to $x$. We denote the union of the sets $E_n(x),
n=1,2,\ldots ,$ by $\E=\E(x)$. Finally if $\beta=a/q\in\Q$ and
$(a,q)=1$ then we define the {\it height} of $\beta$ as
$h(\beta)=|q|$. Our main result is
\begin{theorem}\label{MQestimate2}
Let $g$ be a non-negative arithmetical function which satisfies
$g(m)=O(m^{-(1/2+\gamma)})$ for some $\gamma>0$. For each
$\beta\in\Q$ let $\beta=[a_0;a_1,\ldots,a_L]$ and define
$c(\beta)=g(a_L)$. Then for any $\epsilon>0$ and for almost all
$x\in\R$ we have as $Q\rar\infty$ that
\begin{align*}
\sum_{\substack{\beta\in\E(x) \\
h(\beta)\le Q}}c(\beta)=~&\frac{12}{\pi^2}\left(\sum_{m=1}^
\infty g(m)\log\left(1+\frac{1}{m}\right)\right)\log Q\nonumber\\
&+O_\epsilon\left((\log Q)^{1-\gamma/2}(\log\log
Q)^{5/8+\epsilon-\gamma^2}\right).
\end{align*}
\end{theorem}
Since the intermediate convergents to a real number interpolate
its partial quotients, the function $g$ that we are using here
can be thought of as an average value of the function $f$ which
appears in Khinchin's Theorem. Also the number of terms in our
sum is limited not by a fixed quantity $n$, as in Khinchin's, but
rather by the number of intermediate convergents to each point
$x$ which have small height.

At this point we make a few remarks. First of all we are actually
going to prove something more general than Theorem
\ref{MQestimate2}, namely Theorem \ref{MQestimate1} below, and
then deduce the former as a special case. Secondly since there is
obviously room for more flexibility in the constants $c(\beta)$
than we are allowing in our theorems, we are going to set the
problem up with this in mind before we specialize to the case at
hand. Finally approximation of real numbers by rationals is
essentially unique modulo the integers. For this reason (and also
to be in sympathy with the notation used in \cite{haynes2006})
from now on we prefer to formulate everything with respect to
$\R/\Z$ and $\Q/\Z$ instead of with respect to $\R$ and $\Q$.

Now let us proceed to develop in more detail the results which
are contained in this paper. For each $Q\in\N$ define the Farey
fractions of order $Q$ by
$$\F_Q = \{a/q\in \Q/\Z:1\le q\le Q, \gcd(a,q)=1\},$$
and let $\F$ denote the set of Farey fractions of all orders. For
each $\beta\in\F$ with $h(\beta)\ge 2$ there exist unique
elements $\beta'$ and $\beta''$ in $\F$ with the property that
$\beta'<\beta<\beta''$ are consecutive in the Farey fractions of
order $h(\beta)$. Thus for each $\beta\in\F$ with $h(\beta)\ge 2$
we may define a function $\chi_\beta:\R/\Z\rar\R$ by
\begin{align*}
\chi_\beta(x) & = \begin{cases}1 &\text{ if $x\in (\beta',\beta'')$},\\
    \frac{1}{2} &\text{ if $x=\beta'$ or $x=\beta''$, and}\\
    0 &\text{ if $x\notin [\beta',\beta'']$}.
\end{cases}
\end{align*}
For $\beta=0$ we define the function $\chi_\beta:\R/\Z\rar\R$ by
$\chi_\beta(x)=1.$ In \cite[Theorem 5]{haynes2006} it is shown
that if $x\in(\R/\Z)\setminus\Q$ and $\beta\in\F$ then
\begin{equation}\label{chibeta1}
\chi_\beta(x)=\begin{cases}1 &\text{ if $\beta\in E_n(x)$ for
some $n$, and}\\ 0&\text{otherwise.}
\end{cases}
\end{equation}
In other words $\chi_\beta$ is the indicator function of the
event that an irrational point $x$ has $\beta$ as one of its
intermediate or principal convergents. Now given any sequence of
real constants $\{c(\beta)\}_{\beta\in\F}$ and any positive
integer $Q$ it is natural for our purposes to consider the
function $M_Q:\R/\Z\rar\R$ defined by
\begin{equation}\label{MQeqn1}
M_Q(x)=\sum_{\beta\in\F_Q}c(\beta)\chi_\beta(x).
\end{equation}
By (\ref{chibeta1}) it is apparent that for irrational
$x\in\R/\Z$,
\begin{equation}\label{MQeqn2.1}
M_Q(x)=\sum_{\substack{\beta\in \E(x)\\
h(\beta)\le Q}}c(\beta).
\end{equation}
For example if we set $c(\beta)=1$ for all $\beta$ then we have
for irrational $x$ that
\[M_Q(x)=\#\{\beta\in\F_Q:\beta\in E_n(x)\text{ for some } n\}.\]
It is tempting to try to obtain an almost everywhere asymptotic
formula for $M_Q(x)$ with this choice of constants, but we will
see that in its simplest form this goal is unattainable.

To make this statement more precise, for each point $x\in\R/\Z$
and for each positive integer $Q$ we define the integer $N(Q,x)$
by
\begin{equation*}
N(Q,x)=\min\{n\ge 0 : Q< q_n+q_{n-1}\}.
\end{equation*}
We also let $a(Q,x)$ be the unique integer with the property that
\[a(Q,x)q_{N(Q,x)-1}+q_{N(Q,x)-2}\le Q<(a(Q,x)+1)q_{N(Q,x)-1}+q_{N(Q,x)-2}.\]
Note that by the definition of $N(Q,x)$ we have that $1\le
a(Q,x)\le a_{N(Q,x)}$. Furthermore in this notation we can
rewrite equation (\ref{MQeqn2.1}) as
\begin{equation}\label{MQeqn2}
M_Q(x)=\sum_{n=1}^{N(Q,x)}\sum_{\substack{\beta\in E_n(x) \\
h(\beta)\le Q}}c(\beta).
\end{equation}
In Section \ref{prelimsection} we will show how well known
results about continued fractions can be used to prove the
following result.
\begin{theorem}\label{NQestimate1}
For any $\epsilon>0$ and for almost every $x\in\R/\Z$ we have as
$Q\rar\infty$ that
\begin{equation*}
N(Q,x)=\frac{12\log 2}{\pi^2}\log Q+O_\epsilon\left((\log
Q)^{1/2}(\log\log Q)^{3/2+\epsilon}\right).
\end{equation*}
Consequently for almost every $x\in\R/\Z$ the number of
intermediate convergents to $x$ with height less than or equal to
$Q$ is asymptotic as $Q\rar\infty$ to
\begin{equation}\label{intconvseqn1}
\frac{12}{\pi^2}\log Q(\log\log
Q)\left(1+o(1)\right)+O(M_0(Q,x)),
\end{equation}
where $M_0(Q,x)=\max\{a_1(x),a_2(x),\ldots ,a_{N-1}(x),a(Q,x)\}$.
\end{theorem}
Note that for almost every $x\in\R/\Z$ the quantity $M_0(Q,x)$ is
infinitely often larger than $f(N(Q,x))$ for any function $f$
which satisfies
\[\sum_{n=1}^\infty f(n)^{-1}=\infty.\]
Thus the quantity (\ref{intconvseqn1}) is heavily influenced by
large partial quotients. It is also interesting to compare this
estimate with the expected value for $M_Q(x)$. We leave it to the
reader to verify that for $\beta=a/q\in\F$, $q\ge 2$ we have that
\begin{equation}\label{Echibeta}
E(\chi_\beta)=\frac{1}{h(\beta')h(\beta'')}.
\end{equation}
It should be understood that here and throughout the paper we are
using Lebesgue measure on $\R/\Z$. Using (\ref{Echibeta}) we find
(see \cite[Lemma 6]{haynes2006}) that for $q\ge 2$
\begin{equation}
E\left(\sum_{h(\beta)=q}\chi_\beta\right)= \frac{2
\vphi(q)}{q^2}\left\{\log q + \sum_{p|q} \frac{\log p}{p-1}
    + c_0\right\} + O\left(\frac{\log\log q}{q^2}\right),\label{exsubq2}
\end{equation}
where $c_0$ is Euler's constant, and the sum on the right of
(\ref{exsubq2}) is over prime numbers $p$ that divide $q$. From
this it is not difficult to deduce that the expected value for
the number of intermediate convergents to a real number which
have height $\le Q$ is
\begin{equation*}
\frac{6}{\pi^2}(\log Q)^2+O(\log Q(\log\log Q)).
\end{equation*}
Thus Theorem \ref{NQestimate1} tells us that the almost
everywhere asymptotic behavior of this quantity is significantly
smaller than what one would expect. The interested reader can
compare these observations with the discrete version of this
problem which is treated in \cite{yao1975} and in the last
chapter of \cite{khinchin1964}.

Theorem \ref{NQestimate1} also plays a role in our proof of
Theorem \ref{MQestimate2}, which we will attend to in Section
\ref{MQestsection}. There we will demonstrate how the weak
dependence of partial quotients can be used to prove the
following more general theorem, from which Theorem
\ref{MQestimate2} will be deduced.
\begin{theorem}\label{MQestimate1}
Let $\delta,\epsilon >0$, let $f(n)=\lfloor n(\log
n)^{1/2+\delta}\rfloor$, and let $g:\N\rar\R$ be any non-negative
arithmetical function which satisfies
\begin{align}
\sum_{m=1}^\infty\frac{g(m)}{m}<\infty&~\text{ and}\label{ghypoth1}\\
\limsup_{n\rar\infty}\frac{\sum_{m=1}^{f((n+1)^2)}g(m)}{\sum_{m=1}^{f(n^2)}g(m)}&<\infty.\label{ghypoth2}
\end{align}
For each
$\beta\in\F$ let $\beta=[a_0;a_1,\ldots,a_L]$ and define
$c(\beta)=g(a_L)$. Then for almost every $x\in\R/\Z$ we have as $Q\rar\infty$ that
\begin{align}
M_Q(x)=&\frac{12}{\pi^2}\left(\sum_{m=2}^{f(N)}g(m)\log\left(1+\frac{1}{m}\right)\right)\log Q\label{MQest1main}\\
&+O_\epsilon\left(\left(\sum_{m=1}^{f(N)}g(m)\right)^{1/2}(\log Q)^{3/4}(\log\log Q)^{1/2+\epsilon}\right)\label{MQest1err1}\\
&+O\left(\sum_{m=1}^{N(\log N)^{1+\delta}}g(m)\right),\label{MQest1err2}
\end{align}
where $f(N)=f(N(Q,x))$.
\end{theorem}
Finally we mention here an open problem, also discussed in
\cite{haynes2006}, which has been a motivation for much of this
research. Suppose that $\mc{Q}$ is a collection of positive
integers with the property that
\begin{equation}\label{divcond1}
\sum_{q\in\mc{Q}}\frac{\varphi(q)\log q}{q^2}=+\infty.
\end{equation}
Does it follow that almost every $x\in\R/\Z$ has infinitely many
intermediate convergents $\beta$ with $h(\beta)\in\mc{Q}$? In
other words, does it follow from (\ref{divcond1}) that
\begin{equation}\label{MQdiv1}
\sum_{\substack{\beta\in\F\\h(\beta)\in\mc{Q}}}\chi_\beta(x)=+\infty
\end{equation}
for almost every $x\in\R/\Z$? Condition (\ref{divcond1}) is
easily seen to be necessary in order for (\ref{MQdiv1}) to hold
(see the paragraph preceding \cite[Conjecture 1]{haynes2006}),
but sufficiency seems much more difficult to establish. The
analogous conjecture regarding the principal convergents to
almost all real numbers is a nontrivial theorem of Erd\"os
\cite{erdos1970}. The problem that we are proposing here appears
to share some characteristics with the unknown cases of the
Duffin-Schaeffer Conjecture, a survey of which can be found in
\cite{harman1998}.


\section{Proof of Theorem \ref{NQestimate1}}\label{prelimsection}
We will use the following theorem, which was proved in its
asymptotic form by both Khinchin and Levy in the 1930's.
\begin{KLT}\label{khinlevythm1}
For any $\epsilon>0$ and for almost every $x\in\R/\Z$ we have as
$n\rar\infty$ that
\begin{equation}\label{khinlevythm2}
\log q_n=\frac{\pi^2}{12\log 2}n+O_\epsilon\left(n^{1/2}(\log
n)^{3/2+\epsilon}\right).
\end{equation}
\end{KLT}
The constant $\pi^2/12\log 2$ is the Khinchin-Levy constant, and
will heretofore be referred to as $\gamma_{KL}$. The ingredients
of the proof of this result can be found in Chapter V of
\cite{rockett1992} although we should point out that there is a
typo which persists in the statements of several theorems in that
chapter, including the Khinchin-Levy Theorem. The error term that
we are reporting here can be obtained by combining the proof of
\cite[Theorem V.9.1]{rockett1992} with \cite[Lemma
1.5]{harman1998}.

By combining the Khinchin-Levy Theorem with a couple other
ingredients we are led to the following proof of Theorem
\ref{NQestimate1}.
\begin{proof}[Proof of Theorem \ref{NQestimate1}]
For any positive integer $Q$ and for any irrational $x\in\R/\Z$ we
have that
\[q_{N(Q,x)-1}+q_{N(Q,x)-2}\le Q < q_{N(Q,x)}+q_{N(Q,x)-1},\]
which implies that
\[q_{N(Q,x)-1}\le Q < (a_{N(Q,x)}+1)q_{N(Q,x)-1}+q_{N(Q,x)-2}.\]
Taking the logarithm then yields
\[\log(q_{N(Q,x)-1})\le\log Q<\log(q_{N(Q,x)-1})+\log(a_{N(Q,x)}+1)+\log 2.\]
Now we use the fact that the set \[\{x\in\R/\Z : a_n>n^2\text{
for infinitely many }n\}\] has measure zero \cite[Theorem
V.2.2]{rockett1992}. This means that for almost every $x\in\R/\Z$ we have
\[\log Q=\log q_{N(Q,x)-1}+O(\log N(Q,x)).\]
Since $N(Q,x)=O(\log Q)$ almost everywhere, the proof of the
first part of Theorem \ref{NQestimate1} now follows from the
Khinchin-Levy Theorem.

The second part of Theorem \ref{NQestimate1} follows from the
first part together with a result of H. Diamond and J. Vaaler,
\cite[Corollary 1]{diamond1986}, which states that for almost
every $x\in\R/\Z$ we have as $n\rar\infty$ that
\begin{equation*}
\sum_{k=1}^na_k(x)=\frac{1+o(1)}{\log 2}n\log n+O\left(\max_{1\le
k\le n}a_k(x)\right).
\end{equation*}
Setting $n=N(Q,x)-1$ and incorporating the
extra term $a(Q,x)$ into the error, we obtain exactly what is
reported in the theorem.
\end{proof}


\section{Proofs of Theorems \ref{MQestimate2} and \ref{MQestimate1}}\label{MQestsection}
In this section we will use several classical results from the
metric theory of continued fractions. The most fundamental are
the following two theorems. The first gives the probability
density for the partial quotients in the continued fraction
expansion of numbers in $\R/\Z$, and the second establishes that
the partial quotients in the continued fraction expansions of
almost all real numbers are weakly dependent. For both theorems
we refer the reader to \cite{rockett1992}.
\begin{theorem}\cite[Equation
V.5.2]{rockett1992}\label{anprobdist} Let $n$ and $k$ be positive
integers and let $\mu$ be Lebesgue measure. Then there is a
universal constant $0<q<1$ for which
\begin{equation*}
\mu\{a_n(x)=k\}=\int_{1/(k+1)}^{1/k}\left(\frac{1}{\log
2}\frac{1}{1+x}+O(q^n)\right)dx.
\end{equation*}
\end{theorem}
\begin{theorem}\cite[Theorem
V.7.1]{rockett1992}\label{weakdepthm} Let $n,k,r,$ and $s$ be
positive integers and let $\mu$ denote Lebesgue measure on
$\R/\Z$. Then
\begin{align*}
&\mu\{a_n(x)=r\text{ and }a_{n+k}(x)=s\}\\
&\qquad=\mu\{a_n(x)=r\}\cdot\mu\{a_{n+k}(x)=s\}\cdot(1+O(q^k)),
\end{align*}
where $0<q<1.$
\end{theorem}
Now we discuss some of the key elements which will be needed for
the proof of Theorem \ref{MQestimate1}. For each pair of positive
integers $m$ and $n$ we define a function $f_{m,n}:\R/\Z\rar\R$
by
\begin{equation*}
f_{m,n}(x)=\begin{cases}1&\text{ if }a_n(x)\ge m,\\0&\text{
else}.\end{cases}
\end{equation*}
It is clear from Theorem \ref{anprobdist} that
\begin{align}
\int_{\R/\Z}f_{m,n}~d\mu&=\int_0^{1/m}\left(\frac{1}{\log
2}\frac{1}{1+x}+O(q^n)\right)~dx\nonumber\\
&=\frac{\log\left(1+\frac{1}{m}\right)}{\log
2}+O\left(\frac{q^n}{m}\right)\label{Efmnsum1},
\end{align}
from which it follows that
\begin{equation}
\int_{\R/\Z}\sum_{i=1}^nf_{m,i}~d\mu=\frac{n\log\left(1+\frac{1}{m}\right)}{\log
2}+O\left(\frac{1}{m}\right)\label{Efmnsum2}.
\end{equation}
Using the weak dependence of partial quotients we prove the
following variance estimate.
\begin{lemma}\label{varest1}
Let $m, n_1,$ and $n_2$ be positive integers with $n_1<n_2$. Then
we have that
\begin{equation*}
\var\left(\sum_{i=n_1}^{n_2}f_{m,i}\right)\ll\int_{\R/\Z}\sum_{i=n_1}^{n_2}f_{m,i}~d\mu,
\end{equation*}
and the implied constant is independent of $m, n_1$ and $n_2$.
\end{lemma}
\begin{proof}
Applying Theorem \ref{weakdepthm} we have for $i,j\ge 1$ that
\begin{align}
\int_{\R/\Z}f_{m,i}f_{m,i+j}~d\mu&=\sum_{k=m}^\infty\sum_{\ell=m}^\infty\mu\{a_i(x)=k\text{
and }a_{i+j}(x)=\ell\}\nonumber\\
&=\sum_{k=m}^\infty\sum_{\ell=m}^\infty\mu\{a_i(x)=k\}\cdot\mu\{a_{i+j}(x)=\ell\}\cdot\left(1+O(q^j)\right)\nonumber\\
&=\int_{\R/\Z}f_{m,i}~d\mu\int_{\R/\Z}f_{m,i+j}~d\mu\cdot\left(1+O(q^j)\right).\label{fmnind1}
\end{align}
Then for the variances we have that
\begin{align}
&\var\left(\sum_{i=n_1}^{n_2}f_{m,i}\right)\nonumber\\
&=\sum_{i=n_1}^{n_2}\left(\int_{\R/\Z}f_{m,i}^2~d\mu
-\left(\int_{\R/\Z}f_{m,i}~d\mu\right)^2\right)\nonumber\\
&\quad+2\sum_{i=n_1}^{n_2-1}\sum_{j=1}^{n_2-i}\left(\int_{\R/\Z}f_{m,i}f_{m,i+j}~d\mu
-\int_{\R/\Z}f_{m,i}~d\mu\int_{\R/\Z}f_{m,i+j}~d\mu\right)\nonumber\\
&\ll\int_{\R/\Z}\sum_{i=n_1}^{n_2}f_{m,i}~d\mu
+\sum_{i=n_1}^{n_2-1}\sum_{j=1}^{n_2-i}q^j\int_{\R/\Z}f_{m,i}~d\mu\int_{\R/\Z}f_{m,i+j}~d\mu\label{varfmn1},
\end{align}
where we have used (\ref{fmnind1}) and the facts that
\begin{align*}
f_{m,i}^2(x)=f_{m,i}(x)\quad\text{
and}\quad\int_{\R/\Z}f_{m,i}~d\mu\le 1.
\end{align*}
Now for the sum on $j$ in (\ref{varfmn1}) we use (\ref{Efmnsum1})
to deduce that
\begin{align*}
\sum_{j=1}^{n_2-i}q^j\int_{\R/\Z}f_{m,i+j}~d\mu&=\sum_{j=1}^{n_2-i}\left(\frac{q^j\log\left(1+\frac{1}{m}\right)}{\log
2}+O\left(\frac{q^{i+2j}}{m}\right)\right)\\
&\ll\int_{\R/\Z}f_{m,i}~d\mu,
\end{align*}
and hence that
\begin{align*}
\var\left(\sum_{i=n_1}^{n_2}f_{m,i}\right)&\ll\int_{\R/\Z}\sum_{i=n_1}^{n_2}f_{m,i}~d\mu
+\sum_{i=n_1}^{n_2}\left(\int_{\R/\Z}f_{m,i}~d\mu\right)^2\\
&\ll\int_{\R/\Z}\sum_{i=n_1}^{n_2}f_{m,i}~d\mu.
\end{align*}
Finally note that all of the implied constants depend at most
upon the universal quantity $q$.
\end{proof}
The reader who is familiar with this subject will have already
deduced that Lemma \ref{varest1} can be used to prove an almost
everywhere convergence result. Indeed the following lemma now
follows immediately from \cite[Lemma 1.5]{harman1998}.
\begin{lemma}\label{AEest1}
Let $m$ be a positive integer. For any $\epsilon >0$ and for
almost every $x\in\R/\Z$ we have as $n\rar\infty$ that
\begin{equation*}
\sum_{i=1}^nf_{m,i}(x)=\frac{n\log\left(1+\frac{1}{m}\right)}{\log
2}+O_\epsilon\left(\frac{n^{1/2}(\log
n)^{3/2+\epsilon}}{m^{1/2}}\right).
\end{equation*}
\end{lemma}
It is tempting to try to conclude that the implied constant in
Lemma \ref{AEest1} should depend only on $\epsilon$. If this were
true we could then form sums over $m$ and find a more direct
route to prove Theorem \ref{MQestimate1}. However care must be
taken in passing from the $L^2$-estimate to the almost everywhere
estimate. The proof of the almost everywhere estimate recorded
here requires an application of the convergence part of the
Borel-Cantelli lemma. In applying this lemma the $L^2$-estimates
are used together with Chebyshev's inequality to conclude that
certain events can not happen infinitely often almost everywhere.
However this is not a quantitative statement and in the end it
forces the error term in our almost everywhere estimate to depend
heavily on the functions involved. This observation is the
justification for our choice of proof, for which we will need the
following more technical lemmas.
\begin{lemma}\label{varest2}
Suppose that $g:\N\rar\R$ is a non-negative function and that $M_1, M_2, n_1,$ and $n_2$ are positive integers with $M_1<M_2$ and $n_1<n_2$. Then
\begin{equation*}
\var\left(\sum_{m=M_1}^{M_2}g(m)\sum_{i=n_1}^{n_2}f_{m,i}\right)
\ll\left(\sum_{m=M_1}^{M_2}g(m)\right)\int_{\R/\Z}\sum_{m=M_1}^{M_2}g(m)\sum_{i=n_1}^{n_2}f_{m,i}~d\mu,
\end{equation*}
and the implied constant is independent of $M, n_1,$ and $n_2$.
\end{lemma}
\begin{lemma}\label{varest3}
Suppose that $g, M_1, M_2, n_1,$ and $n_2$ are as in Lemma \ref{varest2}. Write
\begin{align*}
Y_1&=\sum_{m=1}^{M_1}g(m)\sum_{i=n_1+1}^{n_2}f_{m,i},\text{ and}\\
Y_2&=\sum_{m=M_1+1}^{M_2}g(m)\sum_{i=1}^{n_2}f_{m,i}.
\end{align*}
Then for the covariance of $Y_1$ and $Y_2$ we have that
\begin{align*}
\cov\left(Y_1,Y_2\right)\ll\left(\sum_{m=M_1+1}^{M_2}g(m)\right)\int_{\R/\Z}Y_1~d\mu,
\end{align*}
and the implied constant is independent of $M_1, M_2, n_1,$ and $n_2$.
\end{lemma}
The proofs of these lemmas use the same ideas as the proof of
Lemma \ref{varest1} and they therefore rely essentially on
Theorem \ref{weakdepthm}.
\begin{proof}[Proof of Lemma \ref{varest2}]
First note that the argument used to show (\ref{fmnind1}) also shows that for any positive integers $m_1, m_2,$ and $i$ and for any non-zero integer $j\ge 1-i$ we have that
\begin{align*}
\int_{\R/\Z}f_{m_1,i}f_{m_2,i+j}~d\mu&-\int_{\R/\Z}f_{m_1,i}~d\mu
\int_{\R/\Z}f_{m_2,i+j}~d\mu\nonumber\\
&\qquad\ll q^{|j|}\int_{\R/\Z}f_{m_1,i}~d\mu
\int_{\R/\Z}f_{m_2,i+j}~d\mu.
\end{align*}
As before if we sum over $j$ we see that
\begin{align}
\sum_{\substack{j=1-i\\j\not=0}}^\infty\left(\int_{\R/\Z}f_{m_1,i}f_{m_2,i+j}~d\mu\right.&\left.-\int_{\R/\Z}f_{m_1,i}~d\mu
\int_{\R/\Z}f_{m_2,i+j}~d\mu\right)\label{fmnind2}\\
&\ll \int_{\R/\Z}f_{m_1,i}~d\mu\int_{\R/\Z}f_{m_2,i}~d\mu.\nonumber
\end{align}
Using this fact we have that
\begin{align*}
&\var\left(\sum_{m=M_1}^{M_2}g(m)\sum_{i=n_1}^{n_2}f_{m,i}\right)\\
&\qquad=\sum_{m_1,m_2=M_1}^{M_2}g(m_1)g(m_2)\sum_{i=n_1}^{n_2}
\sum_{j=n_1-i}^{n_2-i}\left(\int_{\R/\Z}f_{m_1,i}f_{m_2,i+j}~d\mu\right.\\
&\qquad\qquad\qquad\qquad\qquad\qquad\qquad\qquad\qquad\left.-\int_{\R/\Z}f_{m_1,i}~d\mu\int_{\R/\Z}f_{m_2,i+j}~d\mu\right)\\
&\qquad=\sum_{m_1,m_2=M_1}^{M_2}g(m_1)g(m_2)\left(O\left(\sum_{i=n_1}^{n_2}\int_{\R/\Z}f_{m_1,i}~d\mu
\int_{\R/\Z}f_{m_2,i}~d\mu\right)\right.\\
&\qquad\qquad\qquad\left.+\sum_{i=n_1}^{n_2}\left(\int_{\R/\Z}f_{m_1,i}f_{m_2,i}~d\mu
-\int_{\R/\Z}f_{m_1,i}~d\mu\int_{\R/\Z}f_{m_2,i}~d\mu\right)\right).
\end{align*}
Now since
\begin{align*}
\int_{\R/\Z}f_{m_1,i}f_{m_2,i}~d\mu&=\int_{\R/\Z}f_{\max\{m_1,m_2\},i}~d\mu~\text{ and}\\
\int_{\R/\Z}f_{m_1,i}~d\mu\int_{\R/\Z}f_{m_2,i}~d\mu&\le\int_{\R/\Z}f_{\max\{m_1,m_2\},i}~d\mu
\end{align*}
we have that
\begin{align*}
&\var\left(\sum_{m=M_1}^{M_2}g(m)\sum_{i=n_1}^{n_2}f_{m,i}\right)\\
&\qquad\ll\sum_{m_1,m_2=M_1}^{M_2}g(m_1)g(m_2)\sum_{i=n_1}^{n_2}\int_{\R/\Z}f_{\max\{m_1,m_2\},i}~d\mu\\
&\qquad\le\sum_{m_1,m_2=M_1}^{M_2}g(m_1)g(m_2)\sum_{i=n_1}^{n_2}\int_{\R/\Z}f_{m_1,i}~d\mu\\
&\qquad =\left(\sum_{m=M_1}^{M_2}g(m)\right)\int_{\R/\Z}\sum_{m=M_1}^{M_2}g(m)\sum_{i=n_1}^{n_2}f_{m,i}~d\mu.
\end{align*}
\end{proof}
\begin{proof}[Proof of Lemma \ref{varest3}]
Using (\ref{fmnind2}) we find that $\cov (Y_1,Y_2)$ equals
\begin{align*}
&\sum_{m_1=2}^{M_1}\sum_{m_2=M_1+1}^{M_2}g(m_1)g(m_2)\sum_{i=n_1+1}^{n_2}
\sum_{j=1-i}^{n_2-i}\left(\int_{\R/\Z}f_{m_1,i}f_{m_2,i+j}~d\mu\right.\\
&\qquad\qquad\qquad\qquad\qquad\qquad\qquad\qquad\qquad\left.-\int_{\R/\Z}f_{m_1,i}~d\mu\int_{\R/\Z}f_{m_2,i+j}~d\mu\right)\\
&\quad=\sum_{m_1=2}^{M_1}\sum_{m_2=M_1+1}^{M_2}g(m_1)g(m_2)\left(O\left(\sum_{i=n_1+1}^{n_2}\int_{\R/\Z}f_{m_1,i}~d\mu
\int_{\R/\Z}f_{m_2,i}~d\mu\right)\right.\\
&\qquad\qquad\qquad\left.+\sum_{i=n_1+1}^{n_2}\left(\int_{\R/\Z}f_{m_1,i}f_{m_2,i}~d\mu
-\int_{\R/\Z}f_{m_1,i}~d\mu\int_{\R/\Z}f_{m_2,i}~d\mu\right)\right)\\
&\quad\ll \sum_{m_1=2}^{M_1}\sum_{m_2=M_1+1}^{M_2}g(m_1)g(m_2)\sum_{i=n_1+1}^{n_2}\int_{\R/\Z}f_{\max\{m_1,m_2\},i}~d\mu\\
&\quad\le\sum_{m_1=2}^{M_1}\sum_{m_2=M_1+1}^{M_2}g(m_1)g(m_2)\sum_{i=n_1+1}^{n_2}\int_{\R/\Z}f_{m_1,i}~d\mu\\
&\quad =\left(\sum_{m=M_1+1}^{M_2}g(m)\right)\int_{\R/\Z}\sum_{m=1}^{M_1}g(m)\sum_{i=n_1+1}^{n_2}f_{m,i}~d\mu.
\end{align*}
\end{proof}
Now we use Lemmas \ref{varest2} and \ref{varest3} to prove a crucial almost everywhere result. Suppose that $f:\N\rar\N$ and $g:\N\rar\R$ are non-negative and that $f$ is increasing. For each positive integer $n$ define $X_{n,f}:\R/\Z\rar\R$ by
\begin{equation*}
X_{n,f}(x)=\sum_{m=2}^{f(n)}g(m)\sum_{i=1}^nf_{m,n}(x).
\end{equation*}
Also define $G_f:\N\rar\R$ by
\begin{equation*}
G_f(n)=\left(\sum_{m=1}^{f((n+1)^2)}g(m)\right)\left(\sum_{m=1}^{f(n^2)}g(m)\right)^{-1}.
\end{equation*}
We have the following result.
\begin{lemma}\label{aelemma1}
Let $f, g, X_{n,f},$ and $G_f$ be defined as above and let $\epsilon >0$. Assume that
\begin{align}
\sum_{m=1}^\infty\frac{g(m)}{m}&<\infty~\text{ and}\label{assump1}\\
\sum_{n=1}^\infty \frac{G_f(n)}{n(\log n)^{1+2\epsilon}}&<\infty.\label{assump2}
\end{align}
Then for almost every $x\in\R/\Z$ we have as $n\rar\infty$ that
\begin{align*}
X_{n,f}(x)&=\frac{n}{\log 2}\sum_{m=2}^{f(n)}g(m)\log\left(1+\frac{1}{m}\right)\\
&\qquad+O_\epsilon\left(\left(\sum_{m=1}^{f(n)}g(m)\right)^{1/2}n^{3/4}(\log n)^{1/2+\epsilon}\right).
\end{align*}
\end{lemma}
\begin{proof}
To simplify equations let us write $X_n$ for $X_{n,f}$ and $E(X_n)$ for the expected value of $X_n$ on $\R/\Z$. For each positive integer $i$ let $n_i=i^2$, and for each $n$ let
\[\epsilon(n)=\left(\sum_{m=1}^{f(n)}g(m)\right)^{1/2}E(X_n)^{3/4}(\log E(X_n))^{1/2+\epsilon}.\] By Chebyshev's inequality together with Lemma \ref{varest2} we have that
\begin{align*}
\mu\left\{|X_{n_i}-E(X_{n_i})|>\epsilon(n_i)\right\}
&\ll\frac{\left(\sum_{m=1}^{f(n_i)}g(m)\right)E(X_{n_i})}{\epsilon(n_i)^2}\\
&=\frac{1}{E(X_{n_i})^{1/2}(\log(E(X_{n_i})))^{1+2\epsilon}}\\
&\ll\frac{1}{i(\log i)^{1+2\epsilon}}.
\end{align*}
Thus by the Borel-Cantelli lemma for almost every $x\in\R/\Z$ there are only finitely many $i$ for which \[x\in\left\{|X_{n_i}(x)-E(X_{n_i})(x)|>\epsilon(n_i)\right\}.\]
For the gaps notice that for positive integers $n_2>n_1$ we have
\begin{align*}
X_{n_2}-X_{n_1}&=\sum_{m=2}^{f(n_1)}g(m)\sum_{i=n_1+1}^{n_2}f_{m,i}
+\sum_{m=f(n_1)+1}^{f(n_2)}g(m)\sum_{i=1}^{n_2}f_{m,i}\\
&=S_1+S_2,
\end{align*}
so that by Lemmas \ref{varest2} and \ref{varest3} we obtain
\begin{align*}
\var(X_{n_2}-X_{n_1})&=\var(S_1)+2\cov(S_1,S_2)+\var(S_2)\\
&=O\left(\left(\sum_{m=2}^{f(n_1)}g(m)\right)E(S_1)\right)\\
&\quad+O\left(\left(\sum_{m=f(n_1)+1}^{f(n_2)}g(m)\right)E(S_1)\right)\\
&\quad+O\left(\left(\sum_{m=f(n_1)+1}^{f(n_2)}g(m)\right)E(S_2)\right)\\
&\ll\left(\sum_{m=2}^{f(n_2)}g(m)\right)E(X_{n_2}-X_{n_1}).
\end{align*}
Now we have that
\begin{align}
&\mu\left\{\max_{n_i<n<n_{i+1}}\left|(X_n-X_{n_i})-E(X_n-X_{n_i})\right|>\epsilon (n_i)\right\}\label{gapsum1}\\
&\qquad\le\sum_{n=n_i+1}^{n_{i+1}}\mu\left\{\left|(X_n-X_{n_i})-E(X_n-X_{n_i})\right|>\epsilon (n_i)\right\}\nonumber\\
&\qquad\ll\sum_{n=n_i+1}^{n_{i+1}}\frac{\left(\sum_{m=2}^{f(n)}g(m)\right)E(X_n-X_{n_i})}{\epsilon (n_i)^2}\nonumber\\
&\qquad\le G_f(i)\sum_{n=n_i+1}^{n_{i+1}}\frac{E(X_n-X_{n_i})}{E(X_{n_i})^{3/2}(\log E(X_{n_i}))^{1+2\epsilon}}.\nonumber
\end{align}
Now using (\ref{Efmnsum2}) together with hypothesis (\ref{assump1}) we find that
\begin{align*}
E(X_n)&=\frac{n}{\log 2}\sum_{m=2}^{f(n)}g(m)\log\left(1+\frac{1}{m}\right)+O(1)\\
&=\frac{n}{\log 2}\sum_{m=2}^{\infty}g(m)\log\left(1+\frac{1}{m}\right)\left(1+o(1)\right),
\end{align*}
which gives us
\begin{align*}
(\ref{gapsum1})&\ll G_f(i)\sum_{n=n_i+1}^{n_{i+1}}\frac{n-n_i}{n_i^{3/2}(\log n_i)^{1+2\epsilon}}\\
&\le\frac{G_f(i)(n_{i+1}-n_i)^2}{n_i^{3/2}(\log n_i)^{1+2\epsilon}}\\
&\ll\frac{G_f(i)}{i(\log i)^{1+2\epsilon}}.
\end{align*}
In view of hypothesis (\ref{assump2}) we now have that
\begin{align*}
\sum_{i=1}^{\infty}\mu\left\{\max_{n_i<n<n_{i+1}}\left|(X_n-X_{n_i})-E(X_n-X_{n_i})\right|>\epsilon (n_i)\right\}<\infty
\end{align*}
which together with another application of Borel-Cantelli finishes the proof.
\end{proof}

Next we proceed to the proof of Theorem \ref{MQestimate1}. In our
proof we will use the fact that the partial quotients $a_n$ of
almost every $x\in\R/\Z$ are only finitely often greater than
$n(\log n)^{1+\delta}$. However to obtain the error term we have
reported we will need to use the following refinement of this
fact, which was noticed and proved in \cite{diamond1986}.
\begin{lemma}\cite[Lemma 2]{diamond1986}\label{anmaxlem}
Let $\delta>0$ and for $M\in\N$ set $M'=M(\log M)^{1/2+\delta}$.
For almost all $x\in\R/\Z$ there exist at most finitely many
positive integers $M$ for which the inequalities
\begin{equation*}
a_m>M',\qquad a_n>M'
\end{equation*}
hold for two distinct indices $m,n\le M$.
\end{lemma}
\begin{proof}[Proof of Theorem \ref{MQestimate1}]
For much of the proof allow us to simplify the equations involved
by suppressing the dependence of $N(Q,x)$ upon $Q$ and $x$.

From equation (\ref{MQeqn2}) we have for irrational $x$ that
\begin{align}
M_Q(x)&=\sum_{n=1}^{N-1}\sum_{\beta\in
E_n}c(\beta)+\sum_{\substack{\beta\in E_N\\ h(\beta)\le
Q}}c(\beta)\nonumber\\
&=\sum_{n=1}^{N-1}\sum_{m=1}^{a_n}
c\left(\frac{mp_{n-1}+p_{n-2}}{mq_{n-1}+q_{n-2}}\right)
+\sum_{m=1}^{a(Q,x)}c\left(\frac{mp_{N-1}+p_{N-2}}{mq_{N-1}+q_{N-2}}\right).\label{MQeqn3}
\end{align}
Now for irrational $x$ and for positive integers $n$ and
$1\le m\le a_n$ there are two possibilities for the continued
fraction expansion of the fraction
\[\beta=\frac{mp_{n-1}+p_{n-2}}{mq_{n-1}+q_{n-2}}.\]
If $m\ge 2$ then the unique continued fraction expansion of $\beta$,
as defined in the introduction, is given by
\[\beta=[a_0;a_1,\ldots ,a_{n-1},m],\]
while if $m=1$ we have \[\beta=[a_0;a_1,\ldots ,a_{n-1}+1].\] With
these facts, equation (\ref{MQeqn3}), and the definition of the
constants $c(\beta)$ we have for irrational $x$ that
\begin{align}
M_Q(x)&=\sum_{n=1}^N\left(g(a_{n-1}+1)+\sum_{m=2}^{a_n}g(m)\right)-\sum_{m=a(Q,x)+1}^{a_N}g(m)\nonumber\\
&=g(1)+\sum_{n=1}^{N-1}\sum_{m=2}^{a_n+1}g(m)+\sum_{m=2}^{a(Q,x)}g(m).\nonumber
\end{align}
For fixed $Q$ and $x$ let
\begin{align*}
a_n'(x)=\begin{cases}a_n(x)+1 &\text{ if } 1\le n< N(Q,x), \text{ and}\\a(Q,x)&\text{ if } n=N(Q,x),\end{cases}
\end{align*}
and define
\begin{align*}
M_1(Q,x)&=\max_{1\le n\le N(Q,x)}a_n'(x).
\end{align*}
Then let $n_1=n_1(Q,x)$ be an integer in $\{1,\ldots ,N\}$ with $a_{n_1}'(x)=M_1$ and define
\begin{align*}
M_2(Q,x)&=\max_{\substack{1\le n\le N(Q,x)\\n\not=n_1}}a_n'(x).
\end{align*}
Now for a fixed $x\in\R/\Z$ if $M_1=a(Q,x)$ then we have that
\begin{align*}
M_Q(x)=g(1)+\sum_{m=1}^{M_2}g(m)\sum_{n=2}^{N-1}f_{m,n}(x)+\sum_{m=1}^{M_1}g(m),
\end{align*}
otherwise we have that
\begin{align*}
M_Q(x)=g(1)+\sum_{m=1}^{M_1}g(m)\sum_{n=2}^{N-1}f_{m,n}(x)+\sum_{m=1}^{a(Q,x)}g(m).
\end{align*}
In either case it is clear that
\begin{align}
M_Q(x)-\left(\sum_{m=1}^{M_2}g(m)\sum_{n=2}^{N-1}f_{m,n}(x)\right)\le
g(1)+\sum_{m=1}^{M_1}g(m).\label{MQeqn4}
\end{align}
Now letting
\begin{equation*}
S_Q(x)=\sum_{m=1}^{M_2}g(m)\sum_{n=2}^{N-1}f_{m,n}(x)
\end{equation*}
we may apply Lemma \ref{anmaxlem} to see that for almost every $x\in\R/\Z$ we can choose $Q_0$ large enough that whenever $Q>Q_0$ we have
\begin{align*}
S_Q(x)=\sum_{m=2}^{N(\log N)^{1/2+\delta}}g(m)\sum_{n=1}^{N-1}f_{m,n}(x).
\end{align*}
In other words with $f(N)=\lfloor N(\log N)^{1/2+\delta}\rfloor$ we have that
\begin{equation*}
\lim_{Q\rar\infty}(S_Q(x)-X_{N,f}(x))=0
\end{equation*}
almost everywhere. Now by applying Lemmas \ref{NQestimate1} and \ref{aelemma1}, for any $\epsilon>0$ and for almost every $x\in\R/\Z$ we have as $Q\rar\infty$ that
\begin{align*}
S_Q(x)&=\frac{N(Q,x)}{\log 2}\sum_{m=2}^{f(N)}g(m)\log\left(1+\frac{1}{m}\right)\\
&\qquad+O_\epsilon\left(\left(\sum_{m=1}^{f(N)}g(m)\right)^{1/2}N(Q,x)^{3/4}(\log N(Q,x))^{1/2+\epsilon}\right)\\
&=\frac{12}{\pi^2}\left(\sum_{m=2}^{f(N)}g(m)\log\left(1+\frac{1}{m}\right)\right)\log Q\\
&\qquad+O_\epsilon\left((\log Q)^{1/2}(\log\log Q)^{3/2+\epsilon}\right)\\
&\qquad+O_\epsilon\left(\left(\sum_{m=1}^{f(N)}g(m)\right)^{1/2}(\log Q)^{3/4}(\log\log Q)^{1/2+\epsilon}\right)\\
&=\frac{12}{\pi^2}\left(\sum_{m=2}^{f(N)}g(m)\log\left(1+\frac{1}{m}\right)\right)\log Q\\
&\qquad+O_\epsilon\left(\left(\sum_{m=1}^{f(N)}g(m)\right)^{1/2}(\log Q)^{3/4}(\log\log Q)^{1/2+\epsilon}\right).
\end{align*}
Finally to establish an upper bound for the sums on the right hand side of (\ref{MQeqn4}) we appeal to the fact that for any $\delta>0$ the set
\[\{x\in\R/\Z : a_n(x)>n(\log n)^{1+\delta}\text{ for infinitely many }n\}\]
is a set of measure zero (see \cite[Theorem V.2.2]{rockett1992}). Thus
\begin{align*}
\sum_{m=1}^{M_1}g(m)\le\sum_{m=1}^{N(\log N)^{1+\delta}}g(m)
\end{align*}
almost everywhere as $Q\rar\infty$.
\end{proof}
Finally we come to the proof of our main result.
\begin{proof}[Proof of Theorem \ref{MQestimate2}]
The function $g$ obviously satisfies (\ref{ghypoth1}), so we will
begin by checking that it satisfies (\ref{ghypoth2}). Note that
\begin{equation*}
\frac{\sum_{m=1}^{f((n+1)^2)}g(m)}{\sum_{m=1}^{f(n^2)}g(m)}
=1+\frac{\sum_{m=f(n^2)+1}^{f((n+1)^2)}g(m)}{\sum_{m=1}^{f(n^2)}g(m)},
\end{equation*}
and also that
\begin{align*}
\sum_{m=f(n^2)+1}^{f((n+1)^2)}g(m)&\ll\sum_{m=f(n^2)+1}^{f((n+1)^2)}\frac{1}{m^{1/2+\gamma}}\\
&\ll f((n+1)^2)^{1/2-\gamma}-f(n^2)^{1/2-\gamma}\\
&\ll 2^{(1+\delta)(1/2-\gamma)}\left((n+1)^{1-2\gamma}(\log(n+1))^{(1/2+\delta)(1/2-\gamma)}\right.\\
&\qquad\qquad\qquad\qquad\qquad\qquad \left.-n^{1-2\gamma}(\log n)^{(1/2+\delta)(1/2-\gamma)}\right)\\
&\ll (\log n)^{(1/2+\delta)(1/2-\gamma)}\left((n+1)^{1-2\gamma}-n^{1-2\gamma}\right)\\
&=(\log n)^{(1/2+\delta)(1/2-\gamma)}\left(n^{1-2\gamma}(1+1/n)^{1-2\gamma}-n^{1-2\gamma}\right)\\
&\ll \frac{(\log n)^{(1/2+\delta)(1/2-\gamma)}}{n^{2\gamma}},
\end{align*}
where the last inequality comes from the binomial theorem. Thus
it is clear that
\begin{equation*}
\lim_{n\rar\infty}\frac{\sum_{m=1}^{f((n+1)^2)}g(m)}{\sum_{m=1}^{f(n^2)}g(m)}=1,
\end{equation*}
and the conclusion of Theorem \ref{MQestimate1} may be applied. The error terms (\ref{MQest1err1}) and (\ref{MQest1err2}) can be estimated by
\begin{align*}
\sum_{m=1}^{f(N)}g(m)\ll f(N)^{1/2-\gamma}&\ll(\log Q)^{1/2-\gamma}(\log\log Q)^{(1/2+\delta)(1/2-\gamma)},\text{ and}\\
\sum_{m=1}^{N(\log N)^{1+\delta}}g(m)\ll f(N)^{1/2-\gamma}&\ll(\log Q)^{1/2-\gamma}(\log\log Q)^{(1+\delta)(1/2-\gamma)},
\end{align*}
and the sum in (\ref{MQest1main}) may be extended to infinity, introducing an error of at most a constant times
\begin{align*}
&\left(\sum_{m=f(N)+1}^\infty
g(m)\log\left(1+\frac{1}{m}\right)\right)\log Q
\ll f(N)^{-(1/2+\gamma)}\log Q\\
&\qquad\qquad\qquad\ll (\log Q)^{1/2-\gamma}(\log\log Q)^{(1/2+\delta)(-1/2-\gamma)}.
\end{align*}
Since any value of $\delta>0$ is allowed we may choose
$\delta=\gamma$ to achieve the bounds in the statement of the
theorem.
\end{proof}


\end{document}